\documentclass[a4paper,11pt]{amsart}

\usepackage{amsmath,amsfonts,amssymb,amscd,amsthm,amsbsy,upref}
\usepackage[mathscr]{euscript}
\usepackage{color}
\usepackage{enumerate}
\usepackage{tikz}
\usetikzlibrary{patterns}
\usetikzlibrary{decorations.pathreplacing,angles,quotes}
\usetikzlibrary{arrows}
\usetikzlibrary{calc,automata,patterns,decorations,decorations.pathmorphing}
\usepackage{verbatim}
\usepackage{amssymb,amsthm,amsmath,enumerate}
\usepackage{amsrefs}

\usepackage{stmaryrd}

\definecolor{sealbrown}{rgb}{0.2, 0.08, 0.08}

\usepackage[colorlinks=true]{hyperref}
\hypersetup{urlcolor=blue, citecolor=sealbrown!30!green, linkcolor=blue}

\numberwithin{equation}{section}

\newtheorem{thm}{Theorem}[section]
\newtheorem*{mainthm}{Main Theorem}

\newtheorem{lem}[thm]{Lemma}
\newtheorem{cor}[thm]{Corollary}
\newtheorem{prop}[thm]{Proposition}
\theoremstyle{definition}
\newtheorem*{exA}{Example A}
\newtheorem*{exB}{Example B}
\newtheorem*{exC}{Example C}
\newtheorem*{exD}{Example D}

\newtheorem{dfn}[thm]{Definition}

\theoremstyle{definition}
\newtheorem{rem}[thm]{Remark}

\newcommand{\co}{\mathrm{c}_0}
\newcommand{\coo}{\mathrm{c}_{00}}

\newcommand{\diam}{\ensuremath{\mathrm{diam}}}

\newcommand{\vertiii}[1]{{\left\vert\kern-0.25ex\left\vert\kern-0.25ex\left\vert #1 
    \right\vert\kern-0.25ex\right\vert\kern-0.25ex\right\vert}}
    
    \DeclareFontEncoding{FMS}{}{}
\DeclareFontSubstitution{FMS}{futm}{m}{n}
\DeclareFontEncoding{FMX}{}{}
\DeclareFontSubstitution{FMX}{futm}{m}{n}
\DeclareSymbolFont{fouriersymbols}{FMS}{futm}{m}{n}
\DeclareSymbolFont{fourierlargesymbols}{FMX}{futm}{m}{n}
\DeclareMathDelimiter{\VERT}{\mathord}{fouriersymbols}{152}{fourierlargesymbols}{147}

\begin{document}

\title[On nonexpansiveness]{$\ell_1$ spreading models and the FPP for Ces\`aro mean nonexpansive maps}

\author{C.~S.~Barroso}
\address{C.~S.~Barroso, Department of Mathematics, Federal University of Cear\'a, Fortaleza, CE 60455360, Brazil}
\email{cleonbar@mat.ufc.br}
\thanks{The author was partially supported by CNPq/CAPES}



\maketitle

\begin{abstract} Let $K$ be a nonempty subset of a Banach space $X$. A mapping $T\colon K\to K$ is called $\mathfrak{cm}$-{\it nonexpansive} if for any sequence $(u_i)_{i=1}^\infty$ and $y$ in $K$, $\limsup_{i\to\infty} \sup_{A\subset\{1,\dots, n\}}\|\sum_{k\in A} \big(T u_{i+k} - Ty\big)\|\leq \limsup_{i\to\infty} \sup_{A\subset\{1,\dots, n\}}\|\sum_{k\in A} (u_{i+k} - y)\|$ for all $n\in\mathbb{N}$. As a subclass of the class of nonexpansive maps, its FPP is well-established in a wide variety of spaces. The main result of this paper is a fixed point result relating $\mathfrak{cm}$-nonexpansiveness, $\ell_1$ spreading models and Schauder bases with not-so-large basis constants. As a consequence, we deduce that Banach spaces with the weak Banach-Saks property have the fixed point property for $\mathfrak{cm}$-nonexpansive maps. 
\end{abstract}


\section{Introduction}

\smallskip 

A central question in Banach space theory is to find out which relevant structures can be linked to the geometry of Banach spaces. Structures such as those arising from unconditional bases are of prime relevance, although according to Gowers-Maurey's celebrated result \cite{GM} they are not always available. Despite this, one can always resort to them asymptotically along weakly null sequences. On this concern spreading models have played a crucial role. In \cite{BS}, as e.g. it is proved that every bounded sequence $(x_n)$ in a Banach space $X$ generates a spreading model, that is, there exist a subsequence $(x_{n_i})$, a seminormed space $(E,\vertiii{\cdot})$ ({\it spreading model}) and a sequence $(e_i)\subset E$ such that, for any $n\in\mathbb{N}$ and $\varepsilon>0$, there exists $n_0\in\mathbb{N}$ so that for all $n_0\leq \kappa_1< \kappa_2< \dots < \kappa_n$ and $(a_i)_{i=1}^n\in [ -1, 1]^n$,
\[
\Bigg| \Big\| \sum_{i=1}^n a_i x_{\kappa_i}\Big\| - \Big\VERT\sum_{i=1}^n a_i e_i\Big\VERT\Bigg|< \varepsilon.
\]
The sequence $(e_i)$ ({\it fundamental sequence}) enjoys the $1$-spreading property, that is, it is $1$-equivalent to its subsequences. The unit vector bases of $\mathrm{c}_0$ and $\ell_p$ have this property. It is well known that $1$-spreading weakly null sequences are suppression $1$-unconditional, and so $2$-unconditional. So, as spreading models asymptotically surround bounded sequences it is quite natural to expect them to enjoy the same benefit. As it turns out, this happen when $(x_n)$ is semi-normalized and weakly null \cite[Lemma 2]{Beau}. It is worth noting though that weakly null sequences without unconditional subsequences exists, as witnessed by Maurey and Rosenthal  \cite{M-R}. We refer to the works \cite{AG01, BS, BL, FHHMZ} for more on spreading models and the behavior of semi-normalized weakly null sequences in Banach spaces.

\subsection{Unconditionality and FPP} Asymptotic unconditionality plays a crucial role in metric {\it fixed point problems}. As a sample, Lin \cite[Theorem 3]{Lin} showed that the lack of $\ell_1$ spreading models is crucial for obtaining the FPP in super-reflexive spaces with a   suppression $1$-unconditional basis. Recall that a Banach space $X$ is said to have the FPP (weak-FPP) if for every bounded closed (weakly compact) convex set $C\subset X$, every nonexpansive map $T\colon C\to C$ has a fixed point. The affinities between FPP and semi-normalized weakly null sequences were also studied by García Falset \cite[Theorem 3]{GF97} who proved that $R(X)<2$ implies weak-FPP, where
\[
R(X)=\sup\big\{ \liminf_n \|x_n +x\|\big\}
\] 
with the sup taken over all $x\in B_X$ and all weakly null sequences $(x_n)$ in $B_X$. A major topological property that obstructs $\ell_1$ spreading models is the {\it Banach-Saks property} (BSP) (\cite{Beau, BL}). The space $X$ has BSP if every bounded sequence $(x_n)$ has a subsequence $(y_n)$ such that $(1/n)\sum_{i=1}^n y_i$ is norm-convergent. In its weak form (weak-BSP) it is predicted that every weakly null sequence contains a Ces\`aro norm-null subsequence. Banach and Saks \cite{BaSk} proved that $L_p[0,1]$ ($1<p<\infty$) has BSP. Clearly BSP implies weak-BSP. Also, spaces with BSP are reflexive \cite{NW} and super-reflexive spaces have BSP \cite{Kaku}. BSP and weak-BSP are closely tied to a geometric notion due to Beck \cite{Beck62}: $X$ is called $B$-{\it convex} if for some $n\geq 2$ and $\varepsilon>0$, and all vectors $(x_i)_{i=1}^n$ in $X$ one has $\|(1/n)\sum_{i=1}^n\epsilon_i x_i\|\leq (1- \varepsilon)\max_{i=1,\dots, n}\| x_i\|$ for some choice of signs $(\epsilon_i)_{i=1}^n\in \{-1, 1\}^n$. Beck showed that $B$-convex spaces properly contain all uniformly convex spaces. Thus Hilbert spaces, $\ell_p$ and $L_p[0,1]$ ($1<p<\infty$) are $B$-convex spaces. In contrast, $\ell_1$, $\ell_\infty$ and $\mathrm{c}_0$ are not (see \cite{Giesy66}). Rosenthal \cite{Ros} showed that $B$-convexity implies weak-BSP. It is worth noting that $B$-convexity is equivalent to the lack of $\ell_1$-finite representability \cite{Giesy66}. In particular, super-reflexive spaces are $B$-convex. Further aspects of $B$-convexity including geometric descriptions can be found in \cite{Beck62,James64,Giesy66,GJ73,James74,James78}; see also \cite{GF92, GFFN} for its usefulness in metric fixed point theory.

\subsection{FPP and $\mathfrak{cm}$-nonexpansive maps} One of the most hard problems in metric fixed point theory is to obtain sufficient conditions ensuring the FPP for a class of maps. This problem is even more hard to solve when dealing with nonexpansive maps. There are still several open issues. For example, it is not known whether reflexive (or even super-reflexive) spaces have the FPP for such maps. Garc\'ia Falset \cite{GF92} proved that Banach spaces with weak-BSP and a {\it strongly} bimonotone basis have weak-FPP for nonexpansive maps. The main goal of this paper is to sharply improve this result by considering the class of Ces\`aro mean nonexpansive ($\mathfrak{cm}$-nonexpansive) maps.  

\vskip .1cm 

\begin{dfn} Let $C$ be a bounded and convex subset of a Banach space $X$. We say that a nonexpansive mapping $T\colon C\to C$ is $\mathfrak{cm}$-nonexpansive if
\[
\limsup_{i\to\infty} \sup_{A\subset\{1,\dots, n\}}\Big\|\sum_{k\in A} \big(T u_{i+k} - Ty\big)\Big\|\leq \limsup_{i\to\infty} \sup_{A\subset\{1,\dots, n\}}\Big\|\sum_{k\in A} (u_{i+k} - y)\Big\|,
\]
for all $n\in\mathbb{N}$, $y\in C$ and all sequences $(u_i)_{i=1}^\infty \subset C$. 
\end{dfn}

\newpage

Our main result states that if $X$ fails the weak-FPP for $\mathfrak{cm}$-nonexpansive maps then it has an $\ell_1$-spreading model. In consequence, every Banach space with weak-BSP has the weak-FPP for $\mathfrak{cm}$-nonexpansive maps. Hence the basis requirement is not necessary here. The approach we take to prove this result is inspired by \cite{GF92, Lin}, but relies on some relatively new but implicit technicalities related to minimal sets (see Lemma \ref{lem:3sec3}). 

\vskip .1cm 
\noindent As far as we know, the only results with some degree of similarity to this class of maps are those proved by Amini-Harandi \cite{AminiH} and Dowling \cite{Dowl}. 

\begin{rem} Amini-Harandi \cite{AminiH} called a mapping $T\colon C\to C$ as {\it alternate convexically nonexpansive} ($\mathfrak{ac}$-nonexpansive) if 
\begin{equation}\label{eqn:AminiH}
\Bigg\|\sum_{i=1}^n\frac{(-1)^{i+1}}{n}Tx_i - Ty\Bigg\|\leq \Bigg\|\sum_{i=1}^n\frac{(-1)^{i+1}}{n}x_i - y\Bigg\|,
\end{equation}
for all $n\in\mathbb{N}$ and for all $x_1, x_2, \dots, x_n, y\in C$. He proved that strictly convex spaces have weak-FPP for $\mathfrak{ac}$-nonexpansive maps. Soon after, Dowling \cite{Dowl} developed an improvement that, in addition to covering spaces with the Kadec-Klee property, requires property (\ref{eqn:AminiH}) only for $n=2$. He also showed that Alspach's isometry \cite{Als} is not $2$-$\mathfrak{ac}$-nonexpansive. 
\end{rem}

\begin{rem} It is readily seen that if $T\colon C\to C$ satisfies 
\begin{equation}\label{eqn:1rem3}
\Big\|\sum_{i=1}^n (Tx_i - Ty_i)\Big\|\leq \Big\|\sum_{i=1}^n(x_i - y_i)\Big\|,
\end{equation}
for all sequences $(x_i)_{i=1}^n$ and $(y_i)_{i=1}^n$ in $C$, then $T$ is $\mathfrak{cm}$-nonexpansive. 
\end{rem}


\vskip .1cm 

\begin{exA} It is easily seen that affine nonexpansive maps $T\colon C\to C$ satisfy (\ref{eqn:1rem3}) and hence are $\mathfrak{cm}$-nonexpansive.
\end{exA}

\begin{rem} Recall \cite{CKOW} that a Banach space $X$ is called {\it expand-contract plastic space} ($EC$-space) if its unit ball $B_X$ is plastic, which means to say that every nonexpansive bijection $T\colon B_X\to B_X$ is an isometry. The list of spaces with $EC$-property includes strictly convex spaces, $\mathrm{c}$, $\ell_1$ and many others (see e.g. \cite{CKOW, KZ1, KZ2} and references therein). 
\end{rem}

\vskip .1cm 

\begin{exB} Let $X$ be a $EC$-Banach space. Then every nonexpansive bijection $T\colon B_X\to B_X$ is $\mathfrak{cm}$-nonexpansive. Indeed, since $B_X$ is plastic $T$ is an isometry. Also one has $T(0)=0$ (cf. \cite[Theorem 2.3]{CKOW}). By a result of Mankiewicz \cite{Manki} $T$ extends to a linear isometry in the whole of $X$. Using the linearity of such an extension it is shown that $T$ is $\mathfrak{cm}$-nonexpansive. 
\end{exB}

\vskip .1cm 

\begin{exC} The function $T\colon [0,1]\to [0,1]$ given by $T(t)= \max(t,1/2)$ is $\mathfrak{cm}$-nonexpansive (cf. Appendix \ref{sec:apdx}).
\end{exC}

\vskip .1cm 

\begin{exD} Let $K=\big\{ u\in L_1[0,1] \colon 0\leq u\leq 1\big\}$ and consider the  map $T\colon K\to K$ defined by
\[
Tu(t) = \frac{u(t)}{2}\cdot \int_0^t u(s)ds\quad \mathrm{a.e.\, in}\; [0,1].
\]
Then $T$ is $\mathfrak{cm}$-nonexpansive (cf. Appendix \ref{sec:apdx}). Notice that $K$ is weakly compact and $T(0)=0$. 
\end{exD}

\vskip .1cm 

\subsection{Organization of the paper} We close this section by describing the content of the paper. Section \ref{sec:2} is devoted to background materials where the notation and some basic facts are recalled. In Section \ref{sec:3} we state and prove our main result. Some consequences and final considerations are provided in Sections \ref{sec:4}, \ref{sec:5} and \ref{sec:apdx}. 

\vskip .2cm 

\section{Preliminaries}\label{sec:2}

In this section we set up the background material, referring to \cite{AK, FHHMZ, Singer} for more details. 

\subsection{Notation and terminology} For a Banach space $X$ we denote by $B_X$ its closed unit ball and by $X^*$ its dual space. A sequence $(x_i)$ in $X$ is called  {\it basic} if it is a {\it Schauder} basis for its closed linear span $\llbracket  x_n \rrbracket$. For a Schauder basis $(e_i)$ in $X$ denote by $(e^*_i)\subset X^*$ its coefficient functionals. For $n\in\mathbb{N}$ the $n$-th basis projection is the operator $P_n(x)=\sum_{i=1}^n e^*_i(x)e_i$. The {\it basis constant} of $(e_i)$ is the number $\mathsf{K}_m = \sup_{n\in\mathbb{N}}\|P_n\|$. For $F\subset \mathbb{N}$, the projection $P_F$ over $F$ is the operator $P_F(x) = \sum_{i\in F} e^*_i(x)e_i$. The number $\mathsf{K}_b=\sup_{m<n}\|P_{[m,n]}\|$ is called the {\it bimonotonicity constant} of $(e_i)$. The constant $\mathsf{K}_{sb}=\sup_{m<n}\max( \| P_{[m,n]}\|, \| I - P_{[m,n]}\|)$ is called the {\it strong bimonotonicity projection constant} of $(e_i)$, where $I$ is the identity operator and $[m,n]=\{m, m+1,\dots, n\}$. The basis $(e_i)$ is said to be monotone (resp. bimonotone, strongly bimonotone) if $\mathsf{K}_m=1$ (resp. $\mathsf{K}_b=1$, $\mathsf{K}_{sb}=1$). It is called unconditional if there is  $\mathcal{D}\geq 1$ so that $\| \sum_{i=1}^\infty \epsilon_i a_i e_i\|\leq \mathcal{D}\|\sum_{i=1}^\infty a_i e_i\|$ for all signs $\epsilon_i\in \{\pm 1\}$ and scalars $(a_i)_{i=1}^\infty\in \coo$. The smallest such $\mathcal{D}$ is the {\it unconditional constant} of $(e_i)$. The Banach-Mazur distance between isomorphic Banach spaces $X$ and $Y$ is defined by $\mathrm{d}_{BM}(X, Y)=\inf\big\{ \|\mathcal{L}\| \|\mathcal{L}^{-1}\| \colon \mathcal{L}\colon X\to Y \text{ is a linear isomorphism}\big\}$. Let us recall that $X$ is said to have the {\it bounded approximation property} (BAP) if there exists a $\lambda\geq 1$ so that for every $\varepsilon>0$ and every compact $K\subset X$, one can find an operator $T\in\mathfrak{F}(X)$ so that $\|T\|\leq \lambda$ and $\|Tx -x\|<\varepsilon$ for every $x\in K$, where $\mathfrak{F}(X)$ is the space of finite rank operators on $X$.  

\vskip .2cm 

\subsection{Ultrapower methods and minimal sets} Let $(X,\| \cdot\|)$ be Banach space. We will denote by $[X]$ the quotient space $\ell_\infty(X)/\co(X)$ endowed with the quotient norm given by $\| [v_i] \|:=\limsup_{i\to\infty}\| v_i\|$, where $[v_i]$ denotes the equivalent class whose representative is $(v_i)$. Let $C$ be a weakly compact convex subset of $X$ and assume that $T\colon C\to C$ is a nonexpansive mapping without fixed points. By Zorn's we find a closed convex set $K\subset C$ which is a minimal $T$-invariant set. As usual, the ultra-set $[K]$ of $K$ is defined by $[K]:=\{ [v_i] \in [X] \colon v_i \in K\, \forall i\in \mathbb{N}\}$ and the ultra-mapping $[T]\colon [K] \to [K]$ of $T$ is given by $[T]([v_i])= [T(v_i)]$. Minimal sets have many useful properties (cf. \cite{Ak-K,KSims}) and they are widely used in the study of weak-FPP. One of them is following lemma due to Goebel \cite{Goe} and Karlovitz \cite{Kar}.

\vskip .1cm

\begin{lem}[Goebel-Karlovitz]\label{lem:1sec2} Let $K$ be a minimal weakly compact convex set for a nonexpansive fixed-point free mapping $T$. Then
\[
\lim_{n\to\infty} \| y_n-x\|=\mathrm{diam\,} K
\] 
for all $x\in K$ and any approximate fixed point sequence $(y_n)$ of $T$.
\end{lem}

\medskip 

\noindent Recall that $(y_n)$ approximate fixed points of $T$ if 
\[
\lim_{n\to\infty}\|y_n - Ty_n\|=0. 
\]
Notice in this case that $[T]([y_i])=[y_i]$. The next lemma (due to Lin \cite{Lin}) display another property of minimal sets. 

\begin{lem}[Lin]\label{lem:2sec2} Let $K$ be a weakly compact convex set in a Banach $X$ and $T$ a nonexpansive mapping on $K$. Assume that $K$ is minimal for $T$ and $\mathcal{M}\subseteq [K]$ is a nonempty closed convex subset such that $[T](\mathcal{M})\subseteq \mathcal{M}$. Then 
\[
\sup\big\{ \|[v_i]- x\|\colon [v_i]\in \mathcal{M}\big\}= \mathrm{diam\,} K\;\; \text{ for all } x\in K.
\]
\end{lem}

\noindent For further aspects regarding FPP we refer the reader to \cite{Ak-K, GF-JM-LF, KSims}. 

\medskip 

\section{Main result}\label{sec:3}

\smallskip  

This section is devoted to proving our main result. The proof is based on two main auxiliary lemmas. The first one concerns a weakened notion of bounded approximation (cf. \cite[Definition 18.1]{Singer}). The second concerns a diametral property of minimal sets associated to nonexpansiveness. 

\begin{dfn}[EAB] Let $X$ be a Banach space. We say that $X$ has an extended approximative basis (EAB) if there is a net of finite rank operators $(P_\alpha)_{\alpha\in\mathscr{D}}$ on $X$ such that 
\[
x=\lim_{\alpha\in\mathscr{D}}P_\alpha(x)\,\; (x\in X)\quad\text{and}\quad \sup_{\alpha\in\mathscr{D}} \|P_\alpha\|<\infty.
\] 
If $(P_\alpha)_{\alpha\in\mathscr{D}}$ can be chosen with $\sup_\alpha \|P_\alpha\|\leq \lambda$ for some $\lambda\geq 1$, then $X$ is said to have an extended $\lambda$-approximative basis ($\lambda$-EAB). 
\end{dfn}

\begin{rem} Every separable Banach spaces with BAP has $\lambda$-EAB. 
\end{rem}

\begin{lem}\label{lem:A} Let $X$ be a Banach space with a $\lambda$-EAB $(P_\alpha)_{\alpha\in\mathscr{D}}$. Assume that $K$ is a separable subset of $X$ and $(y_n)$ is a semi-normalized weakly null sequence in $K$. Then there exist an increasing sequence $(\alpha_{m_i})_{i\geq 0}$ in $\mathscr{D}$ and a basic subsequence $(x_{n_i})$ of $(y_n)$ such that  
\begin{itemize}
\item[(\emph{1})] $\lim_{i\to\infty}\big\|P_{\alpha_{m_i}}\big((1/N)\sum_{k=1}^N x_{n_{i+k}}\big)\big\|=0$ for all $N\in\mathbb{N}$. 
\item[(\emph{2})] $\lim_{i\to\infty} \| x_{n_i} - P_{\alpha_{m_i}}(x_{n_i})\|=0$. 
\item[(\emph{3})] $\lim_{i\to\infty} \| x - P_{\alpha_{m_i}}(x)\|=0$ for all $x\in K$. 
\end{itemize}
\end{lem}

\begin{proof} By a result of Bessaga and Pe\l czy\'nski \cite{BP} there is a basic subsequence $(x_{n})$ of $(y_n)$. Let $G=\overline{\textrm{span}}(K)$ (the closed linear span of $K$). Then $G$ is separable. By the proof of Theorem 18.2 in \cite[p. 606]{Singer} there exist an increasing sequence $(\alpha_n)_{n=1}^\infty$ in $\mathscr{D}$ and a separable subspace $F$ of $X$ which contains $G$ and is such that $(P_{\alpha_n})_{n\in\mathbb{N}}$ is an extended $\lambda$-approximative basis for $F$. That is, $\|P_{\alpha_n}\|\leq \lambda$ for all $n\geq 1$ and
\begin{equation}\label{eqn:1lemA}
z = \lim_{n\to\infty} P_{\alpha_n}(z)\quad (z\in F).
\end{equation} 
We now employ the standard gliding hump method. As $(x_n)$ is weakly null and $P_{\alpha_1}$ has finite rank there must exist $n_1\in \mathbb{N}$ so that $\| P_{\alpha_1}(x_n)\|<  1/2^2$ for all $n\geq n_1$. Since $x_{n_1}\in F$ by (\ref{eqn:1lemA}) there is $m_1> n_1$ so that $\| x_{n_1} - P_{\alpha_{m_1}}(x_{n_1})\|<1/2^2$. As $(x_n)_{n\geq m_1}$ is weakly null and $P_{\alpha_{m_1}}$ has finite rank, there is $n_2> m_1$ so that $\| P_{\alpha_{m_1}}(x_n)\|< 1/2^3$ for all $n\geq n_2$. Again by (\ref{eqn:1lemA}) we deduce that $\| x_{n_2} - P_{\alpha_{m_2}}(x_{n_2})\|< 1/2^3$ for some $m_2> n_2$. Continuing in this manner we obtain an increasing sequence $(\alpha_{m_i})_{i\geq 0}$ in $\mathscr{D}$, $m_0=1$, and a subsequence $(x_{n_i})$ such that for all $i\in\mathbb{N}$, $\|P_{\alpha_{m_i}}(x_n)\|< 1/2^{i+1}$ for all $n\geq n_{i+1}$ and $\| x_{n_i} - P_{\alpha_{m_i}}(x_{n_i})\|< 1/2^{i+1}$. The first set of inequalities clearly shows 
\[
\Bigg\| P_{\alpha_{m_i}}\Bigg(\frac{1}{N}\sum_{k=1}^N x_{n_{i+k}}\Bigg)\Bigg\|\leq \frac{1}{2^i}\quad (\forall i, N\in\mathbb{N}),
\]
and this proves (\emph{1}). Also, the second set of inequalities implies (\emph{2}). Finally, we note that (\emph{3}) follows directly from (\ref{eqn:1lemA}). The proof is complete. 
\end{proof}

\vskip .2cm 

\begin{lem}\label{lem:Star1sec3} Let $\mathcal{M}$ be a bounded and convex subset of a normed space $\mathcal{E}$. Assume that $T\colon \mathcal{M}\to \mathcal{M}$ is a $\lambda$-contraction with $\lambda\in (0,1)$. Then there exists a sequence $(x_n)_{n\in\mathbb{N}}$ such that
\[
\| x_n - T(x_n)\|< \frac{1-\lambda}{n}(1+\diam \mathcal{M}),
\]
and
\[
\| x_n - x_m\|\leq \Big( \frac{1}{n} + \frac{1}{m}\Big)(1 + \diam \mathcal{M})\quad\text{for all } n, m\in\mathbb{N}. 
\]
\end{lem}

\begin{proof} Fix $u\in \mathcal{M}$ and for $t\in (0,1)$ define $F_t\colon \mathcal{M}\to \mathcal{M}$ by 
\[
F_t(x) = T(tx + (1-t)u),\quad (x\in \mathcal{M}).
\]
Notice after a direct calculation that 
\[
\|F^n_t x - F^{n+1}_t x\|\leq \lambda^n t^n\diam \mathcal{M}\quad \text{for all } x\in \mathcal{M},\, n\in\mathbb{N}.
\]
Consequently 
\[
\mathrm{d}(F_t, \mathcal{M}):=\inf_{y\in \mathcal{M}} \| y - F_t y\|\leq \lambda^n t^n\diam \mathcal{M},
\]
for all $n\in\mathbb{N}$. It follows that each $F_t$ has null minimal displacement, that is, $\mathrm{d}(F_t,\mathcal{M})=0$ for all $t\in (0,1)$. Thus for any $t, \varepsilon\in (0,1)$ there exists $x_{t,\varepsilon}\in \mathcal{M}$ such that $\| x_{t,\varepsilon} - F_t(x_{t,\varepsilon})\|< \varepsilon$. Notice that this together with triangle inequality implies
\begin{equation}\label{eqn:1lem1}
\| x_{t,\varepsilon} - T(x_{t,\varepsilon}) \| < \varepsilon + (1-t)\diam \mathcal{M}. 
\end{equation}
Now for every $n\in\mathbb{N}$ set
\[
\varepsilon_n:=\frac{ 1-\lambda}{n},\quad t_n:= 1- \frac{1-\lambda}{n}\quad\text{and}\quad x_n:=x_{t_n, \varepsilon_n}.
\]
It follows from (\ref{eqn:1lem1}) that for every $n\in\mathbb{N}$,
\[
\| x_n - T(x_n)\|< \varepsilon_n + (1-t_n)\diam \mathcal{M} = \frac{1-\lambda}{n}(1+\diam \mathcal{M}). 
\]
Hence, for $n, m$ in $\mathbb{N}$ we have
\[
\begin{split}
\| x_n - x_m\|&\leq \| x_n - T(x_n)\| + \| T(x_n) - T(x_m)\| + \| x_m - T(x_m)\|\\[1mm]
&\leq \frac{1-\lambda}{n}(1+\diam \mathcal{M}) + \lambda \| x_n - x_m\| + \frac{1-\lambda}{m}(1+\diam \mathcal{M}),
\end{split}
\]
from which one gets that
\[
\|x_n - x_m\|\leq \Big(\frac{1}{n} + \frac{1}{m}\Big)(1 + \diam \mathcal{M}). 
\]
\end{proof}

\begin{rem} It is an immediate consequence of Lemma \ref{lem:Star1sec3} that if $\mathcal{M}$ is also assumed to be complete then $T$ has a unique fixed point (Banach's contraction principle). 
\end{rem}

\begin{lem}\label{lem:Star2sec3} Let $\mathcal{M}$ be a bounded and convex subset of a normed space $\mathcal{E}$. Assume that $T\colon \mathcal{M}\to \mathcal{M}$ is a nonexpansive mapping. Then there exists a sequence $(y_{n})_{n\in\mathbb{N}}$ in $\mathcal{M}$ such that
\[
\| y_{n} - T(y_{n})\|\leq \frac{1 +\diam \mathcal{M}}{n},
\]
for all $n\in\mathbb{N}$.
\end{lem}

\begin{proof} Fix $x_0\in \mathcal{M}$. For $k\in\mathbb{N}$ define $T_k\colon \mathcal{M}\to \mathcal{M}$ by
\[
T_k(x) = \Big( 1- \frac{1}{k+1}\Big)T(x) + \frac{x_0}{k+1},\quad x\in \mathcal{M}. 
\]
Then $T_k$ is a $\lambda_k$-contraction with $\lambda_k:= 1 - \frac{1}{k+1}$ ($k\in\mathbb{N}$). By Lemma \ref{lem:Star1sec3} there is a sequence $(u_{n,k})_{n\in\mathbb{N}}$ in $K$ such that 
\[
\| u_{n,k} - T_k(u_{n,k})\|< \frac{1 - \lambda_k}{n}(1+\diam \mathcal{M})=\frac{1+\diam \mathcal{M}}{n(k+1)},
\]
and
\[
\| u_{n,k} - u_{m,k}\|\leq \Big( \frac{1}{n} + \frac{1}{m}\Big)(1 + \diam \mathcal{M})\quad\text{for all } n,m\in \mathbb{N}.
\]
Combining the definition of $T_k$ with triangle inequality we have
\[
\| u_{n,k} - T(u_{n,k})\|\leq \frac{1 + \diam \mathcal{M}}{n(k+1)}+ \frac{\diam \mathcal{M}}{k+1}\quad (\forall n,k\in\mathbb{N}).
\]
To conclude the proof it suffices to define $y_{n}:= u_{n,n}$ for all $n\in\mathbb{N}$. 
\end{proof}


\vskip .2cm 


\begin{lem}\label{lem:3sec3} Let $K$ be a minimal weakly compact convex set for a fixed point free $\mathfrak{cm}$-nonexpansive mapping $T\colon K\to K$. Assume that $\mathcal{M}$ is a bounded convex subset of $[K]$. If $\mathcal{M}$ is nonempty and $[T]$-invariant then 
\[
\sup_{[v_i]\in\mathcal{M}}\limsup_{i\to\infty}\sup_{A\subset\{1,\dots, N\}}\Bigg\|\frac{1}{N}\sum_{k\in A}(v_{i+k} - x)\Bigg\|=\diam K
\]
for all $x\in K$ and $N\in\mathbb{N}$.
\end{lem}

\begin{proof} The proof borrows ideas from the proof of Lemma \ref{lem:2sec2}. By Lemma \ref{lem:Star2sec3}, we can select an approximate fixed point sequence for $[T]$ in $\mathcal{M}$, say $(\tilde{w}_s)_{s\in\mathbb{N}}$. Write $\tilde{w}_s = [ w^{(s)}_i]$. Fix $x\in K$ and $N\in\mathbb{N}$. Let's use the notation
\[
\Bigg| \frac{1}{N}\sum_{k=1}^N [w^{(s)}_{i+k}] - x  \Bigg|:= \limsup_{i\to\infty}\sup_{A\subset\{1,\dots, N\}}\Bigg\|\frac{1}{N}\sum_{k\in A}(w_{i+k}^s - x)\Bigg\|.
\]
As in \cite[p. 74]{Ak-K} we will show that $\diam K$ is the only cluster point of the sequence $(|(1/N)\sum_{k=1}^N [w^{(s)}_{i+k}] - x|)_s$. Without loss of generality, we assume that there exists the limit 
\[
d:=\lim_{s\to\infty}\Bigg| \frac{1}{N}\sum_{k=1}^N [w^{(s)}_{i+k}] - x  \Bigg|.
\] 
Write $\delta_s= \|\tilde{w}_s -[T](\tilde{w}_s)\|$ and let $(\varepsilon_j)$ a null sequence in $(0,1)$. Then for every $j\in\mathbb{N}$ there exists an integer $s_j> j$ so that $A^s_j =B^s_j \cap C^s_j$ is infinite for all $s> s_j$, where
\[
B^s_j=\Big\{ m\in\mathbb{N}\colon \sup_{A\subset\{1,\dots, N\}}\Big\|\frac{1}{N}\sum_{k\in A}\big( w^{(s)}_{m+k} - x\big)\Big\|\leq d+ 2\varepsilon_j\Big\}
\]
and
\[
C^s_j=\Big\{  m\in\mathbb{N}\colon  \max_{1\leq k\leq N}\| w^{(s)}_{m+k} - T(w^{(s)}_{m+k})\|\leq \varepsilon_j + \delta_s \Big\}.
\] 
Thus we can find increasing sequences $(s(j))_j$ and $(m(j))_j$ in $\mathbb{N}$ such that 
\begin{equation}\label{eqn:star}
\sup_{A\subset\{1,\dots, N\}}\Big\|\frac{1}{N}\sum_{k\in A}\big( w^{(s(j))}_{m(j)+k} - x\big)\Big\|\leq d+ 2\varepsilon_j
\end{equation}
and
\[
\|  w^{(s(j))}_{m(j)+k} - T(w^{(s(j))}_{m(j)+k})\|\leq \varepsilon_j + \delta_{s(j)},
\]
for all $j\in\mathbb{N}$ and $k=1,\dots, N$. So, each $( w^{(s(j))}_{m(j)+k})_j$ is an approximate fixed point sequence for $T$ in $K$. Passing to a subsequence, we may assume that $( w^{(s(j))}_{m(j)+k})_j$ weakly converges to some $z_k\in K$, for $k=1, \dots, N$. Set 
\[
y_{j,k}:=w^{(s(j))}_{m(j)+k}\quad\text{ for }\,j\in\mathbb{N},\;\; k=1,\dots, N. 
\]
We now claim that 
\begin{equation}\label{eqn:2sec3}
\limsup_{j\to\infty}\sup_{A\subset\{1,\dots, N\}}\Bigg\| \frac{1}{N}\sum_{k\in A}( y_{j,k} - y)\Bigg\|=\diam K\quad\forall y\in K\,\forall N\in\mathbb{N}.
\end{equation}
Indeed, for $N\in\mathbb{N}$ define $\alpha_N\colon K\to \mathbb{R}_+$ by
\[
\alpha_N(y)=\limsup_{j\to\infty} \sup_{A\subset\{1,\dots, N\}}\Bigg\|\frac{1}{N}\sum_{k\in A}(y_{j,k} - y)\Bigg\|\quad y\in K. 
\]
It is not difficult to see that $\alpha_N(Ty)\leq \alpha_N(y)$ for all $y\in K$. Also, notice that $\alpha_N$ is convex. Hence, since $K$ is minimal, $\alpha_N$ is constant. So, for some $c_N\in\mathbb{R}_+$, $\alpha_N(y)=c_N$ for all $y\in K$. Notice that $\big((1/N)\sum_{k=1}^N( y_{j,k}-y)\big)_j$ weakly converges to $(1/N)\sum_{k=1}^N(z_k- y)$. Thus as $\|\cdot\|$ is weak lower semi-continuous we have
\[
\begin{split}
\Bigg\|\frac{1}{N}\sum_{k=1}^N(z_k - y)\Bigg\| &\leq\limsup_{j\to\infty} \Bigg\|\frac{1}{N}\sum_{k=1}^N( y_{j,k} - y)\Bigg\|\\[1mm]
&\leq \limsup_{j\to\infty} \sup_{A\subset\{1,\dots, N\}}\Bigg\|\frac{1}{N}\sum_{k\in A}( y_{j,k} - y)\Bigg\|=c_N.
\end{split}
\]
Therefore
\[
\begin{split}
\diam K=\sup_{y\in K}\Bigg\| \frac{1}{N}\sum_{k=1}^N z_k- y\Bigg\|&\leq c_N\leq \diam K
\end{split}
\]
for all $y\in K$, where in the first equality we used \cite[Theorem 2.4]{Ak-K}. This proves (\ref{eqn:2sec3}). 

\vskip .2cm 
\noindent Therefore by (\ref{eqn:star}) we deduce that
\[
\diam K\leq d
\]
and this easily proves the lemma. 
\end{proof}

\vskip .2cm 

\noindent We are now ready to state and prove the main theorem of this paper. 

\begin{mainthm}\label{thm:M} Let $Z$ be a Banach space having a $\lambda$-EAB $(P_\alpha)_{\alpha\in\mathscr{D}}$ with $\lambda<2$.
Assume that $X$ is a subspace of $Z$ that fails the weak-FPP for $\mathfrak{cm}$-nonexpansive maps. Then $X$ has a spreading model isomorphic to $\ell_1$. 
\end{mainthm}

\begin{proof} Let $\|\cdot\|$ denote the norm of $Z$. By assumption there exist a weakly compact convex set $C\subset X$ and a $\mathfrak{cm}$-nonexpansive map $T\colon C\to C$ with no fixed points. Let $K$ be a minimal weakly compact convex set $K\subset C$ with positive diameter such that $T(K)\subseteq K$. By standard arguments we may assume that $0\in K$, $\diam K=1$ and that there is an approximate fixed point sequence $(x_n)\subset K$ of $T$ weakly converging to $0$. By Lemma \ref{lem:1sec2},
\begin{equation}\label{eqn:3sec3}
\lim_{n\to\infty}\| x_n -y\|=1\quad (\forall\,y\in K).
\end{equation}
Moreover, it is well known that weakly compact convex minimal sets are separable. Then by Lemma \ref{lem:A} and Brunel-Sucheston's result \cite{BS} we find a basic subsequence $(x_{n_i})$ of $(x_n)$ generating a spreading model $\mathbb{E}$, and an increasing sequence $(\alpha_{m_i})$ in $\mathscr{D}$  fulfilling properties (\emph{1})--(\emph{3}). For $i\in\mathbb{N}$ set 
\[
R_i= I - P_{\alpha_{m_i}}.
\] 
\noindent In what follows we shall denote by $\VERT \cdot\VERT$ the norm of $\mathbb{E}$ and by $\mathbf{e}=(e_k)$ its fundamental sequence. Recall that for all scalars $(a_i)_{i=1}^m$ in $\mathbb{R}$ one has
\begin{equation}\label{eqn:5sec3}
\Bigg\VERT \sum_{i=1}^m a_i e_i\Bigg\VERT = \lim_{s_1\to\infty}\,\lim_{s_2\to\infty}\dots \lim_{s_m\to\infty} \Bigg\| \sum_{i=1}^m a_i x_{n_{s_i}}\Bigg\|.
\end{equation}
It follows from this and (\ref{eqn:3sec3}) that $\mathbf{e}$ is normalized. Now pick a constant $d>0$ so that
\begin{equation}\label{eqn:6sec3}
d< \frac{2 - \lambda}{2(1+\lambda)}.
\end{equation}

\vskip .2cm
\noindent We now proceed by following the general ideas of \cite{Lin} and \cite{GF92}, but with a different approach. Fix $m\in\mathbb{N}$ and set

\begin{align*}
\mathcal{M}_m=\left\{  [v_{i}]\in [K]\colon \hskip -.35cm 
\begin{matrix}
&\exists x\in K \textrm{ s.t. }\displaystyle\limsup_{i\to\infty}\sup_{A\subset\{1,\dots, m\}}\Big\|\displaystyle\frac{1}{m}\sum_{k\in A} (v_{i+k}- x)\Big\|\leq d,\\[2.8mm]
& \;\text{ and }\;\displaystyle\frac{1}{m}\limsup_{i\to\infty}\sup_{A\subset\{1,\dots, m\}}\sum_{k\in A}\|v_{i+k} - x_{n_{i+k}}\|\leq \frac{1}{2}
\end{matrix}\right\}.
\end{align*}
Clearly $\mathcal{M}_m$ is bounded and convex. Let's prove that $[T](\mathcal{M}_m)\subset \mathcal{M}_m$. Fix any $[v_i]\in \mathcal{M}_m$ and take $x\in K$ satisfying  
\[
\limsup_{i\to\infty}\sup_{A\subset\{1,\dots, m\}}\Big\|\displaystyle\frac{1}{m}\sum_{k\in A}( v_{i+k}- x)\Big\|\leq d
\]
and
\[
\frac{1}{m}\limsup_{i\to\infty}\sup_{A\subset\{1,\dots, m\}}\sum_{k\in A}\|v_{i+k} - x_{n_{i+k}}\|\leq \frac{1}{2}.
\]
Note that $Tx\in K$ and $[T]([v_i])\in [K]$. Therefore as $T$ is $\mathfrak{cm}$-nonexpansive and $\| x_{n_i} - T(x_{n_i})\|\to 0$ we have 
\[
\limsup_{i\to\infty}\sup_{A\subset\{1,\dots, m\}}\Big\|\displaystyle\frac{1}{m}\sum_{k\in A}\big(Tv_{i+k}- Tx\big)\Big\|\leq d
\]
and
\[
\quad\frac{1}{m}\limsup_{i\to\infty}\sup_{A\subset\{1,\dots, m\}}\sum_{k\in A}\|Tv_{i+k} - x_{n_{i+k}}\|\leq \frac{1}{2}. 
\]
This shows that $[T]([v_i])\in \mathcal{M}_m$ and hence $[T](\mathcal{M}_m)\subset \mathcal{M}_m$, as desired. 

\vskip .2cm 

\noindent 
By construction one has 
\[
\|R_i\|\leq 1 +\lambda\quad\text{and}\quad\|I- R_i\|\leq \lambda,\quad\forall i\in\mathbb{N}. 
\]
In addition, by Lemma \ref{lem:A}-($\emph{1}$) we have
\[
\limsup_{i\to\infty}\sup_{A\subset\{1,\dots, m\}}\Bigg\|\big(I- R_i\big)\Big( \frac{1}{m}\sum_{k\in A} x_{n_{i+k}}\Big)\Bigg\|=0. 
\]
Now fix any $[v_i]\in \mathcal{M}_m$ and take $x\in K$ so that 
\[
\limsup_{i\to\infty}\sup_{A\subset\{1,\dots, m\}}\Big\|\displaystyle\frac{1}{m}\sum_{k\in A}(v_{i+k}- x)\Big\|\leq d.
\] 
Note that $\| R_i(x)\|\to 0$ and
\[
\frac{1}{m}\limsup_{i\to\infty}\sup_{A\subset\{1,\dots, m\}}\sum_{k\in A}\|v_{i+k} - x_{n_{i+k}}\|\leq \frac{1}{2}.
\] 
Hence
\[
\begin{split}
\frac{1}{m}\sum_{k\in A} v_{i+k} &=  R_i\Big(\frac{1}{m}\sum_{k\in A} v_{i+k} \Big) +\big( I - R_i\big)\Big(\frac{1}{m}\sum_{k\in A} v_{i+k}  \Big)\\
&= R_i\Big(\frac{1}{m}\sum_{k\in A}\big(v_{i+k} -x\big)\Big)+ \frac{|A|}{m}R_i(x) +\big(I- R_i\big)\Big( \frac{1}{m}\sum_{k\in A} x_{n_{i+k}}\Big)\\[1mm]
&\hskip 3cm  + \big(I - R_i\big)\Big(\frac{1}{m}\sum_{k\in A} (v_{i+k}   - x_{n_{i+k}})\Big) ,
\end{split}
\]
from which we get 
\[
\begin{split}
\sup_{A\subset\{1,\dots, m\}}\Big\| \frac{1}{m}\sum_{k\in A} v_{i+k} \Big\| &\leq \sup_{A\subset\{1,\dots, m\}}\Big\| R_i\Big(\frac{1}{m}\sum_{k\in A}\big( v_{i+k} -x\big)\Big)\Big\|  + \|R_i(x)\|\\[1mm]
 &\hskip .5cm +\sup_{A\subset\{1,\dots, m\}}\Bigg\|\big(I- R_i\big)\Big( \frac{1}{m}\sum_{k\in A} x_{n_{i+k}}\Big)\Bigg\|\\[1mm]
 &\hskip .5cm + \sup_{A\subset\{1,\dots, m\}}\Bigg\| \big(I - R_i\big)\Big(\frac{1}{m}\sum_{k\in A} (v_{i+k}   - x_{n_{i+k}})\Big)\Bigg\|.
\end{split}
\]
It follows that
\[
\limsup_{i\to\infty} \sup_{A\subset\{1,\dots, m\}}\Big\| \frac{1}{m}\sum_{k\in A} v_{i+k} \Big\| \leq d(1+\lambda) + \frac{\lambda}{2}. 
\]
Consequently, by (\ref{eqn:6sec3}) we have
\[
\sup_{[v_i]\in\mathcal{M}_m}\limsup_{i\to\infty}\sup_{A\subset\{1,\dots, m\}}\Big\| \frac{1}{m}\sum_{k\in A}v_{i+k}\Big\|\leq d(1 + \lambda)  +  \frac{\lambda}{2}
< 1.
\]
Therefore by Lemma \ref{lem:3sec3} we infer that $\mathcal{M}_m=\emptyset$ for all $m\in\mathbb{N}$.

\vskip .2cm 

\noindent As we will see below, the idea behind defining the set $\mathcal{M}_m$ is aligned with the need to force the generation of $\ell_1$ spreading models. To see this, consider the class
\[
[\vartheta_{i}]=\frac{1}{2}[x_{n_{i}}].
\]
It is easy to see that
\[
\frac{1}{m}\limsup_{i\to\infty}\sup_{A\subset\{1,\dots, m\}}\sum_{k\in A}\| \vartheta_{i+k} - x_{n_{i+k}}\|\leq \frac{1}{2}. 
\]
It follows therefore that
\[
2d<\limsup_{i\to\infty}\sup_{A\subset\{1,\dots, m\}}\Bigg\|\frac{1}{m}\sum_{k\in A} x_{n_{i+k}}\Bigg\|. 
\]
It then follows from this, (\ref{eqn:5sec3}) and (\ref{eqn:6sec3}) that there exists a constant $\delta>0$ such that for all $m\in\mathbb{N}$, one can find $A_m\subset\{1,\dots, m\}$, with $|A_m|\to\infty$ as $m\to\infty$, such that
\[
\delta |A_m|< \Bigg\VERT \sum_{k=1}^{|A_m|} e_{k}\Bigg\VERT.
\]
Since $(x_{n_i})$ is semi-normalized and weakly null $(e_i)$ is a non-trivial spreading model. Consequently the previous inequality combined with \cite[Proposition 6.4]{AKT} shows that $(e_i)$ is equivalent to the unit basis of $\ell_1$. This completes the proof of the theorem. 
\end{proof}

\vskip .2cm

\section{Immediate consequences}\label{sec:4}

In this section we explore some consequences of the main result of this work. We first consider the following localized version of the notion of weak Banach-Saks property.

\begin{dfn} A subset $K$ of a Banach space $X$ is said to have the weak-BSP (or, that is a weak-BS set) if whenever $(x_n)_n\subset K$ weakly converges to $x$ then there exists a subsequence $(y_i)_i$ of $(x_n)$ such that the sequence of averages $((1/n)\sum_{i=1}^n y_i)_n$ is norm-convergent to $x$.
\end{dfn} 

\noindent Cleary if $X$ has weak-BSP then every weakly compact subset of $X$ has weak-BSP. 

\vskip .1cm 
\noindent Let us recall (see e.g. \cite[Definition 1]{L-ART}) that if a sequence $(x_n)_n$ weakly converges to some $x\in X$, then it is said to generate an $\ell_1$-spreading model when there is $\delta>0$ such that
\begin{equation}
\Bigg\|\sum_{n\in S} a_n(x_n - x)\Bigg\|\geq \delta \sum_{n\in S}|a_n|
\end{equation}
for every $S\subseteq \mathbb{N}$ with $\# S\leq \min S$ and every sequence of scalars $(a_n)_{n\in S}$. 

\begin{rem} The proof of \cite[Theorem 2.4]{L-ART} shows that if $K$ is a weakly compact set with weak-BSP then no weakly-convergent sequence in $K$ can generate an $\ell_1$-spreading model. 
\end{rem}

\vskip .1cm 

\noindent Therefore our main result can be rephrased as follows:

\begin{thm}[Main Theorem]\label{thm:1} Let $Z$ be a Banach space having a $\lambda$-EAB $(P_\alpha)_{\alpha\in\mathscr{D}}$ with $\lambda<2$.
Assume that $C$ is a weakly compact convex subset of $Z$ that fails the fixed point property for $\mathfrak{cm}$-nonexpansive maps. Then $C$ fails the weak-BSP. 
\end{thm}

\vskip .1cm 

\noindent The first important consequence is the following fixed point theorem. 

\begin{thm}\label{thm:2} Every weakly compact convex weak-BS subset of a Banach space $X$ has the fixed point property for $\mathfrak{cm}$-nonexpansive maps. 
\end{thm}

\begin{proof} Let $C\subset X$ be a weakly compact convex subset with weak-BSP. Since $C$ is separable, so is its closed linear space. Therefore, there is no loss of generality in assuming that $X$ is itself separable. By a result of Banach and Mazur \cite{B} $X$ embeds isometrically into $C[0,1]$ which, as is well known, has a monotone Schauder basis. For $n\in\mathbb{N}$ let $P_n$ denote the corresponding $n$-\textrm{th} basis projection. Clearly $\{P_n\}_{n\in\mathbb{N}}$ is a $1$-EAB for $C[0,1]$. Since weak-BSP is invariant under isometries, $C$ can be seen as a subset of $Z=C[0,1]$ enjoying the weak-BSP. By Theorem \ref{thm:1} the result follows.  
\end{proof}

\vskip .1cm 

\begin{rem} It is worth noting that $L_1[0,1]$ has a monotone basis and (due to Alspach's example \cite{Als}) fails the weak-FPP for isometric maps. In contrast, by a result of Szlenk \cite{Szlenk} $L_1[0,1]$ has the weak-BSP and consequently has the weak-FPP for $\mathfrak{cm}$-nonexpansive maps. As a result $\mathfrak{cm}$-nonexpansiveness is a sharp condition in our result. 
\end{rem}


\noindent An open question in metric fixed point theory is whether super-reflexive Banach spaces have FPP. Our next result settles this problem for the class of $\mathfrak{cm}$-nonexpansive maps. 

\begin{rem} Recall \cite[Definition 3]{James72} that a Banach space $X$ is called {\it super-reflexive} if every Banach space $Y$ that is finitely representable in $X$ is reflexive. Recall also that given Banach spaces $X$ and $Y$, $Y$ is said to be {\it finitely  representable} in $X$ if for any $\varepsilon>0$ and any finite dimensional subspace $N$ of $Y$ there exist a subspace $M$ of $X$ and an isomorphism $\mathfrak{L}\colon N\to M$ such that $(1- \varepsilon)\| y\|\leq \|\mathfrak{L} y\| \leq (1 + \varepsilon) \| y\|$ for all $y\in N$.
\end{rem}

\vskip .1cm 

\begin{thm}\label{thm:5} Let $X$ be a super-reflexive Banach space. Then $X$ has the fixed point property for $\mathfrak{cm}$-nonexpansive maps. 
\end{thm}

\begin{proof} As we have seen before super-reflexive spaces are $B$-convex, and hence have the weak-BSP. By Theorem \ref{thm:2} the result follows. 
\end{proof}

\begin{cor} Let $X$ be a Banach space isomorphic to a super-reflexive space. Then $X$ has the fixed point property for $\mathfrak{cm}$-nonexpansive maps. 
\end{cor}

\begin{proof} Super-reflexivity is invariant under isomorphisms (cf. \cite{James72}). 
\end{proof}

\begin{cor} Every Banach space isomorphic to a uniformly convex space has the fixed point property for $\mathfrak{cm}$-nonexpansive maps. 
\end{cor}

\begin{proof} A uniformly convex space is uniformly non-square and a uniformly non-square space is super-reflexive (cf. \cite{James72}). 
\end{proof}

\noindent The following corollaries follow immediately from above results. 

\begin{cor} Every Banach space isomorphic to $\ell_p$ with $1< p< \infty$ has the fixed point property for $\mathfrak{cm}$-nonexpansive maps. 
\end{cor}

\begin{cor} Every Banach space isomorphic to a Hilbert space has the fixed point property for $\mathfrak{cm}$-nonexpansive maps. 
\end{cor}

\vskip .5cm

\section{Relation to existing results}\label{sec:5}
The FPP has been studied since the 1965s, especially in connection with geometric properties of Banach spaces. The earlier contributions of Browder \cite{Brow65}, G\"ohde \cite{Goh65} and Kirk \cite{Ki65} were determining and inspiring. Note that FPP and weak-FPP are equivalent in reflexive spaces. The  spaces $\ell_p$, $L_p[0,1]$ with $1< p< \infty$ have FPP. In contrast, $\co$, $\ell_1$ and $L_1[0,1]$ fail it. However, they have weak-BSP. Another important result is that a bounded, closed convex subset $K$ of $\co$ has the FPP if and only if $K$ is weakly compact. One direction is due to Maurey \cite{M}, while the other was proved by Dowling, Lennard and Turett \cite{DLT}. Two further significant results are due to Lin \cite{Lin} and Dom\'inguez Benavides \cite{Benavides1}. Lin proved that $\ell_1$ admits an equivalent norm $\VERT\cdot \VERT$ for which $(\ell_1, \VERT\cdot \VERT)$ has the FPP, giving the first example of a nonreflexive Banach space with this property. Dom\'inguez Benavides proved every Banach space which can be embedded in $\co(\Gamma)$ can be equivalently renormed so as to have FPP. However whether reflexive spaces have FPP remains open. Our main result shows that if one intends to use Maurey's ultrapower techniques to attack this problem, a deeper understanding on the impact of $\ell_1$-spreading models is still needed.

\smallskip 

\subsection{Further applications and questions} In the sequel we show how the main theorem connects to existing works. 

\medskip 

\noindent{\bf ($\mathfrak{a}$)} Let $X$ be a Banach space. For $k\in\mathbb{N}$ denote by $s_k(X)$ the supremum of the set of numbers $\varepsilon\in [0,2]$ for which there exist points $x_1, \dots, x_{k+1}$ in $B_X$ with 
\[
\min\big\{ \| x_i - x_j\|: i\neq j\big\}\geq \varepsilon.
\]
In \cite{GFFN} Garc\'ia Falset, Llorens Fuster and Mazcu\~n\'an Navarro introduced the geometric coefficient $\tilde{\varepsilon}_0^k(X)$ given by
\[
\tilde{\varepsilon}_0^k(X)=\sup\Big\{ \varepsilon\in [0, s_k(X)): \tilde{\delta}^k(\varepsilon)=0\Big\}
\]
where $\tilde{\delta}^k\colon [0, s_k(X))\to [0,1]$  is defined as follows:
\[
\tilde{\delta}^k(\varepsilon)=\inf\Bigg\{ 1 - \Big\|\frac{1}{k+1}\sum_{i=1}^{k+1} x_i\Big\|\colon \{x_i\}_{i=1}^{k+1}\subset B_X,\, \min_{i\neq j}\| x_i - x_j\|\geq \varepsilon\Bigg\}.
\]
They then proved that if $X$ has a strongly bimonotone basis and $\tilde{\varepsilon}_0^k(X)< 2$ for some $k\in\mathbb{N}$, then $X$ has the weak-FPP. A careful analysis of their proof shows that the assumption {\it $\tilde{\varepsilon}^k_0(X)<2$ for some $k\in\mathbb{N}$} actually implies $X$ is $B$-convex. Therefore in view of Theorem \ref{thm:2} we have:

\begin{thm} Let $X$ be a Banach space such that $\tilde{\varepsilon}_0^k(X)< 2$ for some $k\in\mathbb{N}$. Then $X$ has the weak-FPP for $\mathfrak{cm}$-nonexpansive maps. 
\end{thm}


\vskip .1cm

\noindent{\bf ($\mathfrak{b}$)} As mentioned before Lin \cite{Lin2} gave the first example of a {\it non}-reflexive Banach space with FPP. Inspired by Lin's approach many authors have provided several other renorming fixed-point techniques (e.g. see \cite{C-SDFJLST}). In \cite{James64} James build an example of a non-reflexive $B$-convex space, solving an open question from \cite{James74}. Theorem \ref{thm:2} shows that the quoted space is an example of a non-reflexive Banach space that has the weak-FPP for $\mathfrak{cm}$-nonexpansive maps without any renorming procedure. 


\vskip .17cm

\noindent{\bf ($\mathfrak{c}$)} $B$-convexity is not equivalent to the weak-FPP. Indeed, by the results of Maurey \cite{M} and Dowling, Lennard and Turett \cite{DLT}, a closed bounded convex subset of $\co$ has the FPP if and only if it is weakly compact. It is also an open question whether $\co$ has the weak-FPP under equivalent renormings. 

\vskip .17cm

\noindent{\bf ($\mathfrak{e}$)} Among the {\it non} $B$-convex reflexive spaces with FPP (see \cite{AM}) are the Ces\`aro sequence spaces $\textit{ces}_p$,$p\in (1,\infty)$, of real sequences $x=(a_k)$ so that 
\[
\| x\|_{c(p)}=\Bigg( \sum_{n=1}^\infty \Big(\frac{1}{n}\sum_{k=1}^n|a_k|\Big)^p\Bigg)^{1/p}<\infty.
\]

\vskip .17cm

\noindent{\bf ($\mathfrak{f}$)} The {\it characteristic of convexity} of a Banach space $X$ is the number 
\[
\varepsilon_0(X)=\sup\{\varepsilon\in [0,2]: \delta_X(\varepsilon)=0\}
\]
where $\delta_X(\varepsilon)$ denotes the usual Clarkson {\it modulus of convexity} of $X$. The space $X$ is said to be {\it uniformly non-square} whenever $\varepsilon_0(X)<2$. In \cite{GF-LF-MN2} Garc\'ia Falset, Llorens Fuster and Mazcu\~n\'an Navarro solved a long-standing problem in the theory by proving that all uniformly non-square Banach spaces have FPP. It is worth pointing out that spaces with such property are super-reflexive.

\vskip .17cm

\noindent{\bf ($\mathfrak{g}$)} Another (apparently still open) problem in the theory asks whether $X\oplus Y$ has the weak-FPP provided that $X$ and $Y$ are two Banach spaces enjoying this property. There are plenty of works addressing this problem (see \cite{W} and references therein). Let $Z$ be a finite dimensional normed space $(\mathbb{R}^n, \| \cdot\|_Z)$. We shall write $(X_1\oplus\dots \oplus X_n)_Z$ for the $Z$-direct sum of the Banach spaces $X_1, \dots, X_n$ equipped with the norm
\[
\| (x_1, \dots, x_n)\|= \| \big( \| x_1\|_{X_1}, \dots, \|x_n\|_{X_n}\big)\|_Z
\]
whenever $x_i\in X_i$ for each $i=1,\dots, n$. In \cite[Lemma 11]{Giesy66}, Giesy proved that $(X_1\oplus\dots \oplus X_n)_{\ell^n_1}$ is $B$-convex if and only if $X_1, \dots, X_n$ are. Thus:

\begin{cor} If $X=(X_1\oplus\dots \oplus X_n)_{\ell^n_1}$ is the $\ell^n_1$-direct sum of $B$-convex Banach spaces then $X$ has the weak-FPP for $\mathfrak{cm}$-nonexpansive maps. 
\end{cor}

\vskip .17cm

\noindent{\bf ($\mathfrak{h}$)} For our next consideration we need a bit of notation. Let $(\Omega, \Sigma, \mu)$ be a complete $\sigma$-finite measure space and let $p\colon \Omega\to [1,+\infty]$ be a measurable function.  Next consider the vector space $\mathcal{X}$ of all measurable functions $g\colon \Omega\to\mathbb{R}$. For $g\in \mathcal{X}$ define the modular 
\[
\rho(g):=\int_{\Omega_f} |g(t)|^{p(t)}d\mu + \mathrm{ess\,sup}_{p^{-1}(\{+\infty\})}|g(t)|,
\]
where $\Omega_f:=\{ t\in \Omega\colon p(t)< \infty\}$. The Variable Lebesgue Space (VLS) (\cite[Definition 2.1]{D-BJ2}) is the space $L^{p(\cdot)}(\Omega)$ endowed with the Luxemburg norm 
\[
\| g\| = \inf\Big\{ \alpha>0 \colon \rho\Big( \frac{ g}{\alpha}\Big)\leq 1\Big\}\,\text{ for } g\in \mathcal{X}_\rho.
\]
In \cite[Theorem 2.5]{D-BJ2} T. Dom\'inguez Benavides and M. Jap\'on established the following relevant characterization of reflexivity for VLSs.

\begin{thm}[Dom\'inguez Benavides-Jap\'on] Let $(\Omega, \Sigma, \mu)$ be an arbitrary $\sigma$-finite measure space and let $p\colon \Omega\to [1,+\infty]$ be a measurable function. The following conditions are equivalent:
\begin{itemize}
\item[\emph{i)}] $L^{p(\cdot)}(\Omega)$ is reflexive.
\item[\emph{ii})] $L^{p(\cdot)}(\Omega)$ contains no isomorphic copy of $\ell_1$. 
\item[\emph{iii)}] Let $\Omega^*:=\Omega\setminus p^{-1}(\{1, +\infty\})$. Then $1< p_{-}(\Omega^*)\leq p_{+}(\Omega^*)< +\infty$ and $p^{-1}(\{1, + \infty\})$ is essentially formed by finitely many atoms at most. 
\end{itemize}
\end{thm}
\noindent Let us remark that their proof that {\it iii)} implies {\it i)} shows a bit more, they in fact proved that $L^{p(\cdot)}(\Omega)$ is isomorphic to a uniformly convex space. Thus, as $B$-convexity is invariant under isomorphisms (\cite[Corollary 6]{Giesy66}), we may deduce that $L^{p(\cdot)}(\Omega)$ is reflexive if and only if it is $B$-convex. As a scholium of this fact by Theorem \ref{thm:2} we have:

\begin{cor} Let $(\Omega, \Sigma, \mu)$ be an arbitrary $\sigma$-finite measure space and let $p\colon \Omega\to [1,+\infty]$ be a measurable function. If $L^{p(\cdot)}(\Omega)$ is reflexive then it has the fixed point property for $\mathfrak{cm}$-nonexpansive maps. 
\end{cor}

\noindent We note that this result yields a slight generalization of \cite[Theorem 4.2]{D-BJ2}, as no further condition is needed on the function $p(\cdot)$. As highlighted in \cite[Theorem 3.3]{D-BJ2}, see also comments in p. 12, one can find plenty of examples of nonreflexive VLSs that still have the weak-FPP. For instance, $L^{1+x}([0,1])$ is one of them. All these remarks naturally lead us to ask whether there is a reflexive {\it non} $B$-convex subspace of $L^{p(\cdot)}(\Omega)$. 

\vskip .17cm

\noindent{\bf ($\mathfrak{i}$)} It is known that $C[0,1]$ fails the weak-BSP. Moreover as $C[0,1]$ is a universal separable Banach space, it fails the weak-FPP for nonexpansive maps. Notice however that $C[0,1]$ contains a weakly compact convex set with weak-BSP (cf. \cite{FT, L-ART}). We do not know whether there is a weakly compact convex set in $C[0,1]$ which fails FPP for $\mathfrak{cm}$-nonexpansive maps.

\vskip .5cm

\section{Appendix}\label{sec:apdx}

\smallskip 

\begin{prop} The function $\varphi \colon [0,1]\to [0,1]$ defined by
\[
\varphi(t)= \max(t,1/2)
\]
is $\mathfrak{cm}$-nonexpansive. 
\end{prop}

\begin{proof} Fix $n\in\mathbb{N}$ and take arbitrary numbers $t_i, s\in [0,1]$, $(i\in\mathbb{N})$. Then
\[
\Bigg| \sum_{k\in A}\big(\varphi(t_{i+k}) - \varphi(s)\big)\Bigg|= \Bigg|\sum_{k\in A}\big( \max(t_{i+k}, 1/2) - \max(s,1/2)\big)\Bigg|:=\Delta_i(A).
\]
Assume without loss of generality that
\[
 \Delta_i(A)= \sum_{k\in A}\big(\max(t_{i+k}, 1/2) - \max(s,1/2)\big).
\]
Note that
\[
\begin{split}
 \Delta_i(A)=& \Bigg(\sum_{k\in \{ t_{i+k}\geq 1/2\}\cap A}\big( t_{i+k} - \max(s,1/2)\big) + \\[1mm]
 &\hskip 1.5cm +\sum_{k\in \{ t_{i+k}< 1/2\}\cap A}\big( 1/2 - \max(s,1/2)\big)\Bigg) \\[1mm]
 &\leq \sum_{k\in \{ t_{i+k}\geq 1/2\}\cap A}( t_{i+k} -s)\\[1mm]
 &= \Bigg| \sum_{k\in \{ t_{i+k}\geq 1/2\}\cap A}( t_{i+k} -s)\Bigg|\leq \sup_{B\subset\{ 1,\dots, n\}}\Bigg|\sum_{k\in B}(t_{i+k} - s)\Bigg|.
\end{split}
\]
Therefore,
\[
\sup_{A\subset\{1,\dots, n\}}\Bigg| \sum_{k\in A}\big(\varphi(t_{i+k}) - \varphi(s)\big)\Bigg|\leq \sup_{B\subset\{ 1,\dots, n\}}\Bigg|\sum_{k\in B}(t_{i+k} - s)\Bigg|
\]
and hence
\[
\limsup_{i\to\infty}\sup_{A\subset\{1,\dots, n\}}\Bigg| \sum_{k\in A}\big(\varphi(t_{i+k}) - \varphi(s)\big)\Bigg|\leq \limsup_{i\in\infty} \sup_{A\subset\{ 1,\dots, n\}}\Bigg|\sum_{k\in A}(t_{i+k} - s)\Bigg|.
\]
\end{proof}


\vskip .2cm 


\begin{prop} Let $K=[\mathbf{0},\mathbf{1}]$ denote the order interval in $L_1[0,1]$, and denote by $\|\cdot\|_1$ the standard $L_1$-norm. The mapping $T\colon K\to K$ defined \textrm{a.e.} in $[0,1]$ by
\[ 
Tu(t) = \frac{1}{2}u(t)\cdot\int_0^t u(s)ds 
\]
is $\mathfrak{cm}$-nonexpansive.
\end{prop} 
\begin{proof} Indeed, let $n\in\mathbb{N}$ be fixed and let $v\in K$ and $(u_i)_{i=1}^\infty$ be arbitrary functions in $K$. Fix any set $A\subset \{1, \dots, n\}$. Then 
\[
\begin{split}
2\Big\| \sum_{k\in A}( Tu_{i+k} - Tv)\Big\|_1 &= 2\int_0^1 \Big|  \sum_{k\in A}( Tu_{i+k} - Tv)(t)\Big| dt\\[1mm]
&\leq \int_0^1\Big|\sum_{k\in A} u_{i+k}(t)\cdot \int_0^t (u_{i+k}(s) - v(s))ds \Big|dt  \\[1mm]
&\hskip 1.7cm + \int_0^1\Big|\sum_{k\in A}\big( u_{i+k}(t)- v(t)\big)\cdot \int_0^t v(s)ds\Big| dt.
\end{split}
\]
\paragraph{\bf Claim 1} 
\[
\int_0^1\Big|\sum_{k\in A}\big( u_{i+k}(t)- v(t)\big)\cdot \int_0^t v(s)ds\Big| dt\leq 
\sup_{A\subset\{1,\dots, n\}}\Big\|\sum_{k\in A}\big(u_{i+k} - v\big)\Big\|_1.
\]
This follows directly from the fact that $0\leq v\leq 1$ on $[0,1]$. 


\vskip .3cm 


\paragraph{\bf Claim 2}
\[
\int_0^1\Big|\sum_{k\in A} u_{i+k}(t)\cdot \int_0^t (u_{i+k}(s) - v(s))ds \Big|dt\leq \sup_{C\subset\{1,\dots, n\}}\Big\|\sum_{k\in C}(u_{i+k} - v)\Big\|_1.
\]
To see this, for $t\in [0,1]$ set
\[
\Delta^1_t:= \sum_{k\in A} u_{i+k}(t)\cdot \int_0^t (u_{i+k}(s) - v(s))ds\quad\text{and}\quad \Delta^2_t: =-\Delta^1_t.
\]
Then
\[
\int_0^1\Big|\sum_{k\in A} u_{i+k}(t)\cdot \int_0^t (u_{i+k}(s) - v(s))ds \Big|dt=\int_{\{t\colon \Delta^1_t\geq 0\}} \Delta^1_t dt   + \int_{\{t\colon \Delta^1_t<0\}} \Delta^2_t dt.
\]
Note that
\[
\Delta^1_t = \sum_{k\in B} u_{i+k}(t)\int_0^t ( u_{i+k}(s)- v(s))ds + \sum_{k\in A\setminus B}u_{i+k}(t)\int_0^t ( u_{i+k}(s)- v(s))ds,
\]
where $B=\{ k\in A\colon \int_0^t (u_{i+k}(s) - v(s))\geq 0\}$. Similarly we have
\[
\Delta^2_t = \sum_{k\in B} u_{i+k}(t)\int_0^t ( v(s) - u_{i+k}(s))ds + \sum_{k\in A\setminus B}u_{i+k}(t)\int_0^t ( v(s)- u_{i+k}(s))ds.
\]
It follows that
\[
\Delta^1_t \leq \max_{1\leq k\leq n} u_{i+k}(t)  \sum_{k\in B} \int_0^t(u_{i+k}(s) - v(s))ds 
\]
and
\[
\Delta^2_t \leq  \max_{1\leq k\leq n} u_{i+k}(t) \sum_{k\in A\setminus B} \int_0^t (v(s) - u_{i+k}(s))ds.
\]
These estimates imply
\[
\Delta^1_t \leq \max_{1\leq k\leq n} u_{i+k}(t)\int_0^t\Big| \sum_{k\in B} (u_{i+k}(s) - v(s))\Big|ds\leq \sup_{C\subset \{1,\dots, n\}}\Big\| \sum_{k\in C}(u_{i+k} - v)\Big\|_1
\]
and (similarly)
\[
\Delta^2_t\leq \max_{1\leq k\leq n} u_{i+k}(t)\int_0^t\Big|\sum_{k\in A\setminus B}^n(v(s) - u_{i+k}(s))\Big|ds\leq \sup_{C\subset\{1,\dots, n\}}\Big\|\sum_{k\in C}(u_{i+k}-v)\Big\|_1.
\]
Consequently
\[
\begin{split}
\int_0^1\Big|\int_0^t \sum_{i=1}^n v_i(t)\cdot (u_i - v_i)ds\Big|dt&=\int_{\{t\colon \Delta^1_t\geq 0\}} \Delta^1_t dt   + \int_{\{t\colon \Delta^1_t<0\}} \Delta^2_t dt\\[1mm]
&\leq \sup_{C\subset\{1,\dots, n\}}\Big\| \sum_{k\in C}(u_{i+k} - v)\Big\|_1 
\end{split}
\]
which proves the Claim 2. 

\vskip .2cm 
\noindent It follows therefore from Claim 1 and Claim 2 that
\[
\Big\| \sum_{k\in A}\big( T(u_{i+k}) - T(v)\big)\Big\|_1\leq \sup_{C\subset\{1,\dots, n\}}\Big\|\sum_{k\in C}(u_{i+k}-v)\Big\|_1.
\]
This proves that $T$ is $\mathfrak{cm}$-nonexpansive. 
\end{proof}


\medskip 


\medskip\noindent
{\bf Acknowledgements.} We warmly thank Prof. B. Turett who pointed out a serious flaw in the first version of this manuscript which was first written in 2023. His considerations led us to envision the current version of the work. The author would like to dedicate this work to the memory of W. A. Kirk, whom he had the pleasure of meeting at the JMM-AMS conference in Baltimore, Maryland, USA, 2019.

\end{document}